\documentclass[12pt,epsfig,amsfonts]{amsart} 
\usepackage{amsmath,amsthm,amssymb,amscd,epsfig,color}
\usepackage{graphicx}
\usepackage{mathrsfs}
\usepackage{ulem}

\newtheorem{prop}{Proposition}[section]
\newtheorem{lemma}[prop]{Lemma}
\newtheorem{thm}[prop]{Theorem}

\theoremstyle{definition}

\theoremstyle{remark}

\numberwithin{equation}{section}

\usepackage[utf8]{inputenc}
\usepackage{fancyvrb}

\begin{document}

\author{Hiroki Takahasi}

\address{ Keio Institute of Pure and Applied Sciences (KiPAS), Department of Mathematics,
Keio University, Yokohama,
223-8522, JAPAN} 
\email{hiroki@math.keio.ac.jp}


\title[Hausdorff dimension of specification for the $(\alpha,\beta)$-shifts]
{Hausdorff dimension of specification\\
for the $(\alpha,\beta)$-shifts}

\begin{abstract}
Specification is an important concept in dynamical systems introduced by Bowen.
Schmeling [Ergod. Th. $\&$ Dynam. Sys.
{\bf 17} (1997), 675--694] proved that the set of $\beta>1$ such that the corresponding $\beta$-shift has specification is of Hausdorff dimension $1$. Hu et al. [Publ. Math. Debr.
{\bf 91} (2017), 123--131] proved that the set of $\beta>1$ such that the corresponding $(-\beta)$-shift has specification is of Hausdorff dimension $1$. 
We show that the set of $(\alpha,\beta)\in[0,1)\times(1,\infty)$
such that
the corresponding $(\alpha,\beta)$-shift has specification is of Hausdorff dimension $2$. A new difficulty 
is a simultaneous control of two critical symbol sequences that determine the ambient shift space. We achieve this by taking 
intersections of two thick Cantor sets in parameter space.
\end{abstract}
\maketitle
\section{Introduction}
Specification is an important concept 
in dynamical systems introduced by Bowen \cite{Bow71}.  
It 
means that one can glue together a  collection of orbit segments to form one orbit. This property is useful in constructions of various types of orbits and invariant measures with prescribed properties that can be effectively used to analyze the dynamics, see \cite{DGS,KH} for example.


Since specification is a very strong property, a natural question is how often it holds. For a wide class of interval maps with discontinuities and associated shift spaces, the answer is negative \cite{Buz97} in terms of the Lebesgue measure in parameter space.
Three prominent examples with the abundance of number-theoretic applications are the following:
\begin{itemize}
\item (the $\beta$-transformation \cite{R57}) $ x\in[0,1)\mapsto \beta x-\lfloor\beta x\rfloor\in[0,1)$;
\item (the $(-\beta)$-transformation \cite{IS09,LS12})
$ x\in(0,1]\mapsto-\beta x+\lfloor\beta x\rfloor+1\in(0,1]$;
\item (the $(\alpha,\beta)$-transformation \cite{Par64})
$x\in[0,1)\mapsto \beta x+\alpha-\lfloor\beta x+\alpha\rfloor\in[0,1)$, 
\end{itemize}
where $0\leq\alpha<1$ and $\beta>1$, and $\lfloor y\rfloor$ for $y\geq0$ denotes the largest integer not exceeding $y$.
Using the partitions of the intervals 
 into the maximal subintervals of continuity, one can code
the dynamics of the transformations into symbolic dynamics.
Let $\Sigma_\beta$, $\Sigma_{-\beta}$, $\Sigma_{\alpha,\beta}$ denote the corresponding shift spaces, called {\it the $\beta$-shift}, {\it the $(-\beta)$-shift}, {\it the $(\alpha,\beta)$-shift} respectively.
From the general result of Buzzi \cite[Theorem~1.5]{Buz97} the following hold: the set of $\beta>1$ such that  $\Sigma_\beta$ has specification is of zero Lebesgue measure; the set of $\beta>1$ such that  $\Sigma_{-\beta}$ has specification is of zero
Lebesgue measure; for each $\alpha\in[0,1)$
the set of $\beta>1$ such that $\Sigma_{\alpha,\beta}$ has specification is of zero Lebesgue measure.
By Fubini's theorem, the set of $(\alpha,\beta)\in [0,1)\times(1,\infty)$ such that $\Sigma_{\alpha,\beta}$ has specification is of zero Lebesgue measure. 

The next natural question is the Hausdorff dimension of specification. 
Schmeling \cite[Theorem~A]{Sch97} proved that the set of $\beta>1$ such that $\Sigma_\beta$ has specification is of Hausdorff dimension $1$.
Hu et al. \cite[Theorem~1.1]{HHY17} proved that the set of $\beta>1$ such that $\Sigma_{-\beta}$ has specification is of Hausdorff dimension $1$.
Extending the argument in \cite{HHY17},
Oguchi and Shinoda \cite[Theorem~1.1]{OS24}  constructed a countably infinite family of $C^\infty$ functions  $\alpha_n\colon(1,\infty)\to[0,1)$ $(n\in\mathbb N)$ such that for every $n\in\mathbb N$, the set of $\beta>1$ such that $\Sigma_{\alpha_n(\beta),\beta}$ has specification is of Hausdorff dimension $1$. 
A main result of this paper is the following theorem on the $(\alpha,\beta)$-shifts.

 \begin{thm}\label{mainthm0}The set of $(\alpha,\beta)\in[0,1)\times(1,\infty)$ such that $\Sigma_{\alpha,\beta}$ has the specification property is of Hausdorff dimension $2$.\end{thm}
 
Although the definition of specification property used in the above three papers \cite{HHY17,OS24,Sch97} and in Theorem~\ref{mainthm0} is different from that in \cite{Buz97}, they are actually equivalent in shift spaces over finite alphabets \cite{KLO16}. 
See \S\ref{spec-sec} for more details and clarifications.







The above three shift spaces are determined by
critical symbol sequences.
For the $\beta$-shifts, it is the sequence corresponding to the virtual orbit of $1$ \cite[Theorem~3]{Par60}.
For the $(-\beta)$-shifts, it is the sequence corresponding to the orbit of $1$ \cite[Theorem~11]{IS09}.
The $(\alpha,\beta)$-shifts are determined by two sequences, one corresponding to the orbit of $0$ and the other to the virtual orbit of $1$ \cite[Theorem~2]{Hof80}.
A necessary and sufficient condition for specification is given in terms of these sequences, see
\cite[Th\'eor\`eme~II]{BM86},
\cite[Lemma~3.2]{HHY17}, \cite[Theorem~1.5]{CLR21} 
for the $\beta$-shifts, the $(-\beta)$-shifts, $(\alpha,\beta)$-shifts respectively. 
An estimation of the Hausdorff dimension of specification then amounts to 
the construction of a large parameter set corresponding to shift spaces for which this condition  is satisfied \cite{HHY17,OS24,Sch97}. 

 In a proof of Theorem~\ref{mainthm0} we proceed along this line. A new difficulty
is a simultaneous
control of the two critical symbol sequences. 
We achieve this by combining the result of Hunt et al. \cite{HKY93} on intersections of thick Cantor sets and
 Newhouse's lower bound \cite{New79} on the Hausdorff dimension of Cantor sets in terms of thickness. We recall these ingredients in \S2. In \S3 we complete the proof of Theorem~\ref{mainthm0}.

\section{Preliminaries}
This section summarizes main ingredients for the proof of Theorem~\ref{mainthm0}.
In \S\ref{spec-sec} we give the definition of specification property and add some clarifications. In \S\ref{alphabeta} we precisely define the $(\alpha,\beta)$-shift based on an induction algorithm for $(\alpha,\beta)$-expansion.
In \S\ref{char-sec} we recall the result of Carapezza et al. \cite[Theorem~1.5]{CLR21} on a characterization of the $(\alpha,\beta)$-shifts and specification in terms of the two critical symbol sequences. 
In \S\ref{thick} we recall the notion of thickness of Cantor sets on the real line, and the results of Hunt et al. \cite{HKY93} and
Newhouse \cite{New79}.

\subsection{Specifications on shift spaces}\label{spec-sec}
Let $\ell\in\mathbb N$ and 
let $\Sigma_\ell$ denote the full shift space $\{0,\ldots,\ell\}^\mathbb N$ on $\ell+1$ symbols. 
 We endow $\Sigma_\ell$ with the product topology of the discrete topology on $\{0,\ldots,\ell\}$. 
  A shift-invariant closed subset of $\Sigma_\ell$ is called a {\it subshift}. 
 A string $w=w_1w_2\cdots w_n$ of elements of $\{0,\ldots,\ell\}$ is called a {\it word} of length $n$. 
We introduce an {\it empty word} $\emptyset$ by the rules
  $\emptyset w=w\emptyset=w$ for any word $w$.
  Let $|w|$ denote
  the word length of a word $w$, and set the word length of the empty word to be $0$. 
  For a subshift $\Sigma$, let $\mathcal L(\Sigma)$ denote the collection of words 
that appear in some elements of $\Sigma$. 

Although
specification can be defined for general dynamical systems on metric spaces, here we restrict ourselves to subshifts. There are some discrepancies in the existing definitions of specification. Clarifications are necessary to correctly interpret our main result and earlier related ones. 

Let $\Sigma$ be a subshift. 
We say $\Sigma$
has   {\it the periodic specification property} if there is an integer $t\geq0$ 
such that 
for every integer $k\geq2$ and all $v^1,\ldots,v^k\in\mathcal L(\Sigma)$, there are $w^1,\ldots, w^{k}\in\mathcal L(\Sigma)$ such that
$v^1w^1v^2w^2
\cdots v^kw^k\in\mathcal L(\Sigma)$ and $|w^i|= t$ for $i=1,\ldots,k$, and 
$[v^1w^1v^2w^2
\cdots v^kw^k]$ contains a periodic point of period $|v^1w^1v^2w^2
\cdots v^kw^k|$.
If we drop the existence of a periodic point from the above definition, that is, if there is an integer $t\geq0$ such that for all $u,v\in\mathcal L(\Sigma)$ there is $w\in\mathcal L(\Sigma)$ such that $uwv\in\mathcal L(\Sigma)$ and $|w|=t$, then
we say $\Sigma$ has   {\it the specification property}. 


From \cite[Theorem~1.5]{Buz97} it follows that:  for each $\alpha\in[0,1)$
the set of $\beta$ such that $\Sigma_{\alpha,\beta}$ has the periodic specification property is of zero Lebesgue measure; the set of $\beta$ such that  $\Sigma_{-\beta}$ has the periodic specification property is of zero Lebesgue measure.
By 
\cite[Th\'eor\`eme~II]{BM86} 
and \cite[Theorem~E]{Sch97},
the set of $\beta$ such that $\Sigma_{\beta}$ has the specification property is of zero Lebesgue measure.
The set of $\beta$ such that $\Sigma_{\beta}$ has the specification property is of Hausdorff dimension $1$
\cite[Theorem~A]{Sch97}.
 The set of $\beta$ such that $\Sigma_{-\beta}$ has the specification property is of Hausdorff dimension $1$ \cite[Theorem~1.1]{HHY17}. The Hausdorff dimension of the set of $(\alpha,\beta)$ such that $\Sigma_{\alpha,\beta}$ has the specification property was analyzed in
\cite[Theorem~1.1]{OS24}.

In fact, the specification property and the periodic specification property are equivalent \cite[\S2]{KLO16}. Therefore, \cite[Theorem~A]{Sch97} \cite[Theorem~1.1]{HHY17} \cite[Theorem~1.1]{OS24} and Theorem~\ref{mainthm0} on the Hausdorff dimension of specification complement \cite[Theorem~1.5]{Buz97} on the Lebesgue measure of specification.
\subsection{The $(\alpha,\beta)$-shifts}\label{alphabeta}
Let $P=[0,1)\times(1,\infty)$, and for each $(\alpha,\beta)\in P$ let $T_{\alpha,\beta}\colon[0,1)\to[0,1)$ denote the corresponding $(\alpha,\beta)$-transformation.
For each $\ell\in\mathbb N$ define
\[E_\ell=\{(\alpha,\beta)\in P\colon\lfloor\alpha+\beta\rfloor=\ell\}=\{(\alpha,\beta)\in P\colon\ell\leq\alpha+\beta<\ell+1\}.\]
We have $P=\bigsqcup_{\ell=1}^\infty E_\ell$.

Let $\ell\in\mathbb N$ and  $(\alpha,\beta)\in E_\ell$. For simplicity we assume $\ell\geq2$.
For each $x\in[0,1)$
write $e_{\alpha,\beta,1}(x)=\lfloor\beta x+\alpha\rfloor$, and
$e_{\alpha,\beta,k+1}(x)=e_{\alpha,\beta,1}(T_{\alpha,\beta}^k(x))$
for $k\in\mathbb N$. 
The integers $e_{\alpha,\beta,k}(x)$ in $\{0,1,\ldots,\ell\}$ form a sequence that encodes the orbit of $x$ under the iteration of $T_{\alpha,\beta}$
into a sequence in 
$\Sigma_\ell$. Define a family $(I_{\alpha,\beta}^m)_{m=0}^{\ell}$ of pairwise disjoint subintervals of $[0,1)$ as follows:
 \[\begin{split}I_{\alpha,\beta}^0=\left[0,\frac{1-\alpha}{\beta}\right),\ I_{\alpha,\beta}^j&=\left[\frac{j-\alpha}{\beta},\frac{j+1-\alpha}{\beta}\right) \
 \text{ for }j=1,\ldots,\ell-1,\\
 I_{\alpha,\beta}^\ell
 &=\left[\frac{\ell-\alpha}{\beta},1\right).\end{split}\]
We have $e_{\alpha,\beta,n}(x)=m$ if and only if $T_{\alpha,\beta}^{n-1}(x)\in I^m_{\alpha,\beta}$.
For all $x\in[0,1)$ we have
\[x=\frac{e_{\alpha,\beta,1}(x)-\alpha+T_{\alpha,\beta}(x)}{\beta},\]
and thus
\[T_{\alpha,\beta}(x)=\frac{e_{\alpha,\beta,1}(T_{\alpha,\beta}(x))-\alpha+T_{\alpha,\beta}^2(x)}{\beta}.\]
Plugging the last equality into the previous one gives
\[x=\frac{e_{\alpha,\beta,1}(x)-\alpha}{\beta}+\frac{e_{\alpha,\beta,2}(x)-\alpha}{\beta^2}+\frac{T_{\alpha,\beta}^2(x)}{\beta^2}.\]
Repeating this procedure we obtain {\it the $(\alpha,\beta)$-expansion} of $x$:
\begin{equation}\label{expansion}x=\sum_{n=1}^\infty\frac{e_{\alpha,\beta,n}(x)-\alpha}{\beta^n}=\sum_{n=1}^\infty\frac{e_{\alpha,\beta,n}(x)}{\beta^n}-\frac{\alpha}{\beta-1}.\end{equation}
 Let \[\Sigma_{\alpha,\beta}=\overline{\{(e_{\alpha,\beta,n}(x))_{n=1}^\infty\in\Sigma_\ell\colon x\in[0,1)\}},\] where the bar denotes the closure operation in
$\Sigma_\ell$.
The space $\Sigma_{\alpha,\beta}$ is shift-invariant, and called {\it the $(\alpha,\beta)$-shift}.

\subsection{Characterizations of shift space and specification}\label{char-sec}
Let $\ell\in\mathbb N$.
The lexicographical order $\preceq$ in $\Sigma_\ell$
is the total order given by: (i) $\omega\preceq\omega$ for all $\omega\in\Sigma_\ell$; (ii) for distinct $\omega=(\omega_n)_{n=1}^\infty$, $\eta=(\eta_n)_{n=1}^\infty\in\Sigma_\ell$, $\omega\preceq\eta$ if $\omega_s<\eta_s$ where $s=\min\{n\geq1\colon\omega_n\neq\eta_n\}$.

Let $(\alpha,\beta)\in E_\ell$.
The $(\alpha,\beta)$-shift is characterized by the lexicographical order in $\Sigma_\ell$ and the $(\alpha,\beta)$-expansions of the endpoints. Let \[u_{\alpha,\beta}=(e_{\alpha,\beta,n}(0))_{n=1}^\infty
\ \text{ and }\ 
v_{\alpha,\beta}=\lim_{x\nearrow1}(e_{\alpha,\beta,n}(x))_{n=1}^\infty.\] Since $x\in[0,1)\mapsto (e_{\alpha,\beta,n}(x))_{n=1}^\infty$ is monotone increasing, this limit exists. 
Let
  $\sigma$ denote the left shift  acting 
on $\Sigma_\ell$: $(\sigma\omega)_n=\omega_{n+1}$ for $n\in\mathbb N$. We have
\[\Sigma_{\alpha,\beta}=\{\omega\in\Sigma_\ell\colon u_{\alpha,\beta}\preceq\sigma^{n-1}\omega\preceq v_{\alpha,\beta}\text{ for every }n\in\mathbb N \},\]
see \cite[Theorem~2]{Hof80}.
A necessary and sufficient condition for the specification of $\Sigma_{\alpha,\beta}$ is given by $u_{\alpha,\beta}$ and $v_{\alpha,\beta}$. 
For $\omega=(\omega_n)_{n=1}^\infty\in\Sigma_\ell$ and $j,k\in\mathbb N$ with $j\leq k$, write $\omega^{[j,k]}$
for $\omega_j\cdots\omega_k$. Define
\[K(u_{\alpha,\beta})=\{n\in\mathbb N\colon v_{\alpha,\beta}^{[1,n]}=u_{\alpha,\beta}^{[1+j,n+j]}\ \text{ for some }j\in\mathbb N\},\]
\[K(v_{\alpha,\beta})=\{n\in\mathbb N\colon u_{\alpha,\beta}^{[1,n]}=v_{\alpha,\beta}^{[1+j,n+j]}\ \text{ for some }j\in\mathbb N\}.\]

\begin{thm}[{\cite[Theorem~1.5]{CLR21}}]\label{thm-CLR}Let $\alpha\in[0,1)$ and $\beta>2$.
Then $\Sigma_{\alpha,\beta}$ has the specification property if and only if both $K(u_{\alpha,\beta})$ and $K(v_{\alpha,\beta})$ are finite sets.\end{thm}
Since $T_{\alpha,\beta}$ acts on $\Sigma_{\alpha,\beta}$ as the left shift,
$K(u_{\alpha,\beta})$ being finite means that the orbit of $0$ under $T_{\alpha,\beta}$ does not accumulate on $1$. Similarly,
$K(v_{\alpha,\beta})$ being finite means that the orbit of $\lim_{x\nearrow1}T_{\alpha,\beta}(x)$ under $T_{\alpha,\beta}$ does not accumulate on $0$.

\subsection{Intersection of thick Cantor sets}\label{thick}
Intersections of two Cantor sets in the real line naturally appear in dynamical systems (see e.g., \cite{New70,New79,PalTak93}) and number theory (see e.g., \cite{Hal47,Mor98}).
Newhouse \cite{New70} introduced 
the notion of thickness to analyze intersections of two Cantor sets.

We adopt the definition of thickness
by Palis and Takens \cite{PalTak93} that is equivalent 
to the one by Newhouse \cite{New70}.
For a bounded interval $I\subset\mathbb R$ let $|I|$ denote its Euclidean length.
Let $S$ be a Cantor set in $\mathbb R$.
A {\it gap} of $S$ is a connected component of $\mathbb R\setminus S$.
A {\it bounded gap} is a gap which is bounded.
Let $G$ be any bounded gap and $x$ be a boundary point of $G$.
Let $I$ denote the {\it bridge of $S$ at $x$},
i.e., the maximal interval in $\mathbb R$ that satisfies
 $x\in\partial I$, and contains no point of a gap 
    whose Euclidean length is at least $|G|$.
The thickness of $S$ at $x$ is defined by
\[\tau(S,x)=\frac{|I|}{|G|}.\] The {\it thickness} $\tau(S)$ of $S$
is the infimum of $\tau(S,x)$ over all boundary points $x$ of bounded gaps. Clearly, thickness is preserved under affine maps on $\mathbb R$. Thickness can be used to estimate from below the Hausdorff dimension of Cantor sets in $\mathbb R$.
Let $\dim$ denote the Hausdorff dimension on the Euclidean space $\mathbb R^d$, $d=1$ or $d=2$.
\begin{prop}[{\cite[p.107]{New79}, \cite[p.77, Proposition~5]{PalTak93}}]\label{lower} Let $S\subset\mathbb R$ be a Cantor set with $\tau(S)>0$. Then we have
\[\dim S\geq\frac{\log2}{\log(2+1/\tau(S))}.\]\end{prop}

We say two Cantor sets $S_1$, $S_2$ in $\mathbb R$
are {\it interleaved} if neither set is contained
in the closure of a gap of the other set.
The well-known gap lemma \cite{New70} asserts that two interleaved Cantor sets on the real line intersect each other if the product of their thicknesses is greater than one.
It does not imply any lower bound of the Hausdorff dimension of the intersection of the two Cantor sets. Indeed,
Williams \cite{Wil91} observed that two interleaved Cantor sets can have thicknesses well above $1$ and still only intersect at a single point. 
The next theorem in \cite{HKY93} asserts that
 the intersection of two interleaved Cantor sets with large thicknesses contains a Cantor set with  large thickness.

\begin{thm}[{\rm in \cite[p.881, Remark]{HKY93}}]\label{cant-inter}
For any $\varepsilon\in(0,1)$ there is $M>0$ such that if two Cantor sets $S_1$, $S_2$ in $\mathbb R$ 
with $\tau(S_1)> M$, $\tau(S_2)> M$ are interleaved,
then $S_1\cap S_2$ contains a Cantor set whose thickness is at least
$(1-\varepsilon)\sqrt{\min\{\tau(S_1),\tau(S_2)\}}$.
\end{thm}
Combining Proposition~\ref{lower} and Theorem~\ref{cant-inter}, one can estimate the Hausdorff dimension of intersection of two interleaved Cantor sets with large thickness.



\section{The proof of Theorem~\ref{mainthm0}}
The proof of Theorem~\ref{mainthm0} breaks into three steps. In \S\ref{construct}, for each fixed $\beta$ we construct two sets of $\alpha\in[0,1)$ whose intersection corresponds to the $(\alpha,\beta)$-shifts satisfying the condition in Theorem~\ref{thm-CLR}. In \S\ref{exist-sec} we show that 
under certain conditions on $\beta$, these two sets contain affine copies of the same Cantor set.
In \S\ref{dim-est} we estimate the Hausdorff dimension of the intersection of these affine copies, and complete the proof of Theorem~\ref{mainthm0}.

\subsection{The definition of two parameter sets}\label{construct}
 Throughout this section we assume $\ell\geq3.$ Let
  $\Sigma_\ell^*$ denote the subspace $\{1,\ldots,\ell-1\}^{\mathbb N}$ of $\Sigma_\ell$.
For $(\alpha,\beta)\in E_\ell$ define
\[\Lambda_{\alpha,\beta}=\{x\in[0,1)\colon (e_{\alpha,\beta,n}(x))_{n=1}^\infty\in\Sigma_\ell^* \}.\]
It is easy to see that 
$T_{\alpha,\beta}(\Lambda_{\alpha,\beta})=\Lambda_{\alpha,\beta}$, and $T_{\alpha,\beta}$ acts on $\Lambda_{\alpha,\beta}$ as the restriction of the left shift $\sigma$
to $\Sigma_\ell^*$.
Since $\ell\geq3$, $\Lambda_{\alpha,\beta}$ is a Cantor set. 
For each $\omega\in \Sigma_\ell^*$,
let $x_{\alpha,\beta}(\omega)$ denote the point in $\Lambda_{\alpha,\beta}$ whose symbol sequence is $\omega$.
By \eqref{expansion} we have
\begin{equation}\label{formula0}x_{\alpha,\beta}(\omega)=\sum_{n=1}^\infty\frac{\omega_n-\alpha}{\beta^n}.\end{equation}
In particular, $\Lambda_{\alpha,\beta}$ is a translation of $\Lambda_{0,\beta}$.

Let us record two estimates for $\alpha=0$. For all $\omega\in\Sigma_\ell^*$ we have 
\begin{equation}\label{x}\frac{1}{\beta}\leq x_{0,\beta}(\omega)<\frac{\lfloor\beta\rfloor}{\beta}.\end{equation}
The thickness of $\Lambda_{0,\beta}$ can be immediately computed and evaluated as follows: \begin{equation}\label{tau-thick}\tau(\Lambda_{0,\beta})=\frac{\sum_{j=1}^{\ell-1}|I^j_{0,\beta}|}{|I^0_{0,\beta}|+|I^\ell_{0,\beta}|}=
\frac{\lfloor\beta\rfloor-1}{1-\lfloor\beta\rfloor+\beta}>\frac{\lfloor\beta\rfloor-1}{2}\geq\frac{\ell-2}{2}.\end{equation}

For each $\beta\in(\ell-1,\ell+1)$, let $E_\ell(\beta)=\{\alpha\in[0,1)\colon(\alpha,\beta)\in E_\ell\}$. We have
\begin{equation}\label{el}E_\ell(\beta)=\begin{cases}[\ell-\beta,1)&\text{ if }\beta\leq\ell,\\
[0,\ell+1-\beta)&\text{ if }\ell<\beta.\end{cases}\end{equation}
Define \[\begin{split}R_{\beta}&=\{\alpha\in E_\ell(\beta)\colon x_{\alpha,\beta}(\omega)=T_{\alpha,\beta}(0)\text{ for some }\omega\in\Sigma_\ell^*\},\\
S_{\beta}&=\{\alpha\in E_\ell(\beta)\colon x_{\alpha,\beta}(\omega)=\lim_{x\nearrow 1}T_{\alpha,\beta}(x)\text{ for some }\omega\in\Sigma_\ell^*\}.\end{split}\] 

\begin{lemma}\label{spec-lem}
Let $\beta\in(\ell-1,\ell+1)$. If $\alpha\in R_\beta\cap S_\beta$
then $\Sigma_{\alpha,\beta}$ has the specification property. \end{lemma}
\begin{proof}
If $\alpha\in R_\beta$ then $T_{\alpha,\beta}(0)\in\Lambda_{\alpha,\beta}$, and so 
 $T_{\alpha,\beta}^{n}(0)\in\Lambda_{\alpha,\beta}$
for every $n\in\mathbb N$, which yields $K(v_{\alpha,\beta})=\emptyset$. Similarly,
if $\alpha\in S_\beta$ then $\lim_{x\nearrow 1}T_{\alpha,\beta}(x)\in\Lambda_{\alpha,\beta}$, and so 
 $\lim_{x\nearrow 1}T_{\alpha,\beta}^{n}(x)\in\Lambda_{\alpha,\beta}$
for every $n\in\mathbb N$, which yields $K(u_{\alpha,\beta})=\emptyset$.
Then the assertion of the lemma follows from
Theorem~\ref{thm-CLR}.
\end{proof}


\subsection{Existence of affine copies of $\Lambda_{0,\beta}$}\label{exist-sec}
We show that both $R_\beta$ and $S_\beta$ contain an affine copy of $\Lambda_{0,\beta}$ under certain conditions on $\beta$.

\begin{lemma}\label{Cantor1}
For all $\beta\in(\ell-1,\ell+1)$ 
the following statements hold:
\begin{itemize}\item[(a)] we have \[R_{\beta}=\left\{\frac{\beta-1}{\beta}x_{0,\beta}(\omega)\colon\omega\in \Sigma_\ell^*\right\}\cap E_\ell(\beta);\]
\item[(b)] if $\beta\leq\ell$ and $R_\beta\neq\emptyset$, then $R_\beta$ is a singleton or a Cantor set.
\end{itemize}\end{lemma}
\begin{proof}
From \eqref{formula0}, for all $\omega\in\Sigma_\ell^*$ and all $(\alpha,\beta)\in E_\ell$ we have
\begin{equation}\label{formula1}x_{\alpha,\beta}(\omega)=x_{0,\beta}(\omega )-\frac{\alpha}{\beta-1}.\end{equation} 

Let $\beta\in(\ell-1,\ell+1)$. 
By $T_{\alpha,\beta}(0)=\alpha-\lfloor\alpha\rfloor=\alpha$ and \eqref{formula1}, 
if $\alpha\in R_\beta$ then there exists $\omega\in\Sigma_\ell^*$ such that 
$\alpha=x_{0,\beta}(\omega )-\frac{\alpha}{\beta-1},$
or equivalently $\alpha=\frac{\beta-1}{\beta}x_{0,\beta}(\omega)$.
Conversely, for any $\omega\in\Sigma_\ell^*$ the number $\alpha=\frac{\beta-1}{\beta}x_{0,\beta}(\omega)$ satisfies the equation
\[\alpha=x_{0,\beta}(\omega )-\frac{\alpha}{\beta-1}.\]
Hence, if $\alpha\in E_\ell(\beta)$
then $\alpha\in R_\beta$. This verifies Lemma~\ref{Cantor1}(a).

Since $\ell\geq3$, the first set in the right-hand side of the equality in Lemma~\ref{Cantor1}(a) is a Cantor set not containing $1$. By \eqref{el}, if $\beta\leq\ell$ then we have $E_\ell=[\ell-\beta,1)$.
Hence Lemma~\ref{Cantor1}(b) follows.
\end{proof}

To prove an analogous lemma on $S_\beta$, 
define
\[\tilde S_{\beta}=\left\{\frac{\beta-1}{\beta}(x_{0,\beta}(\omega)+1-\beta+\lfloor\beta\rfloor)\colon\omega\in \Sigma_\ell^*\right\}\cap[ 1-\beta+\lfloor\beta\rfloor,\min\{\ell+1-\beta,1\}).\]
\begin{lemma}\label{Cantor2}
For all $\beta\in(\ell-1,\ell+1)$ the following statements hold:
\begin{itemize}
\item[(a)] 
$\tilde S_\beta$ is contained in $S_{\beta}$;
\item[(b)] if $\beta\leq\ell$ and \begin{equation}\label{cantor-eq1}1-\beta+\lfloor\beta\rfloor\leq\frac{\beta-1}{\beta}\left(\frac{1}{\beta}+1-\beta+\lfloor\beta\rfloor\right),\end{equation}
and
\begin{equation}\label{cantor-eq2}\frac{\beta-1}{\beta}\left(\frac{\lfloor\beta\rfloor}{\beta}+1-\beta+\lfloor\beta\rfloor\right)<1,\end{equation} then \[\tilde S_\beta=\left\{\frac{\beta-1}{\beta}(x_{0,\beta}(\omega)+1-\beta+\lfloor\beta\rfloor)\colon\omega\in \Sigma_\ell^*\right\}.\]
\end{itemize}
\end{lemma}
\begin{proof}

Let $\alpha\in \tilde S_\beta$. There exists $\omega\in\Sigma_\ell^*$ such that
\[\alpha=\frac{\beta-1}{\beta}(x_{0,\beta}(\omega)+1-\beta+\lfloor\beta\rfloor)\]
or equivalently
\begin{equation}\label{formula2}x_{0,\beta}(\omega )-\frac{\alpha}{\beta-1}=\beta+\alpha-1-\lfloor\beta\rfloor.\end{equation}
Since $1-\beta+\lfloor\beta\rfloor\leq\alpha$
we have $1+\lfloor\beta\rfloor=\lfloor\beta+\alpha\rfloor$, or equivalently
\begin{equation}\label{formula3}\beta+\alpha-1-\lfloor\beta\rfloor=\beta+\alpha-\lfloor\beta+\alpha\rfloor.\end{equation}
Clearly we have
\begin{equation}\label{formula4}\beta+\alpha-\lfloor\beta+\alpha\rfloor=\lim_{x\nearrow1}T_{\alpha,\beta}(x).\end{equation}
From \eqref{formula1}, \eqref{formula2}, \eqref{formula3}, \eqref{formula4} we obtain
\[x_{0,\beta}(\omega )-\frac{\alpha}{\beta-1}=\beta+\alpha-\lfloor\beta+\alpha\rfloor=\lim_{x\nearrow1}T_{\alpha,\beta}(x),\]
namely $\alpha\in S_\beta$. 
This verifies Lemma~\ref{Cantor2}(a).

If $\beta\leq\ell$ then we have $\min\{\ell+1-\beta,1\}=1$. Then \eqref{x}, \eqref{cantor-eq1} and \eqref{cantor-eq2} together imply that
the first set in the definition of $\tilde S_\beta$ is contained in
$[1-\beta+\lfloor\beta\rfloor,1)$. This verifies Lemma~\ref{Cantor2}(b).
\end{proof}

\subsection{Estimate of Hausdorff dimension}\label{dim-est}
Let $\varepsilon\in(0,1)$.
For all $\beta\in(\ell-\varepsilon,\ell]$ we have
\begin{equation}\label{epsilon}
0\leq1-\beta+\lfloor\beta\rfloor<\varepsilon.\end{equation}
We assume $\varepsilon$ is sufficiently small depending on $\ell$ so that 
for all $\beta\in(\ell-\varepsilon,\ell]$, \eqref{cantor-eq1}, \eqref{cantor-eq2} hold and in addition
\begin{equation}\label{beta-1}\ell-\beta<\frac{\beta-1}{\beta}\frac{1}{\beta}.\end{equation}
Note that $\varepsilon\to0$ as $\ell\to\infty$.

Let $\beta\in(\ell-\varepsilon,\ell]$. 
By \eqref{x}, \eqref{beta-1} and
Lemma~\ref{Cantor1} we have
\[R_{\beta}=\left\{\frac{\beta-1}{\beta}x_{0,\beta}(\omega)\colon\omega\in \Sigma_\ell^*\right\}.\]
Then $R_\beta$ is an affine copy of the Cantor set $\Lambda_{0,\beta}$. Clearly the convex hull of $R_\beta$ converges to $[0,1]$ in the Hausdorff topology as $\ell\to\infty$. 
By Lemma~\ref{Cantor2}(b),
$\tilde S_\beta$ is an affine copy of $\Lambda_{0,\beta}$, and by
 \eqref{epsilon}, the convex hull of $\tilde S_\beta$ converges to $[0,1]$ in the Hausdorff topology as $\ell\to\infty$. Consequently, if $\ell$ is sufficiently large
 then for all $\beta\in(\ell-\varepsilon,\ell]$, $R_\beta$ and $\tilde S_\beta$ are interleaved Cantor sets of the same thickness $\tau(\Lambda_{0,\beta})$.
By Theorem~\ref{cant-inter}, $R_\beta\cap \tilde S_\beta$ contains a Cantor set with thickness at least $(1/2)\sqrt{\tau(\Lambda_{0,\beta}) }
$, which is bounded from below by $\sqrt{(\ell-2) /8 }$
by \eqref{tau-thick}.
By Proposition~\ref{lower}, 
we get
\[\dim(R_\beta\cap \tilde S_\beta)\geq\frac{\log2}{\log(2+\sqrt{8/(\ell-2) })}.\]
By \cite[Corollary~7.12]{Fal}, 
we obtain
\[\dim\{(\alpha,\beta)\in E_\ell\colon \beta\in(\ell-\varepsilon,\ell], \alpha\in R_{\beta}\cap \tilde S_{\beta}\}\geq\frac{\log2}{\log(2+\sqrt{8/(\ell-2) })}+1.\]
From this estimate and Lemma~\ref{spec-lem} and
 Lemma~\ref{Cantor2}(a),  
 it follows that 
\[\lim_{\ell\to\infty}\dim\{(\alpha,\beta)\in E_\ell\colon\text{$\Sigma_{\alpha,\beta}$ has the specification property}\}=2.\]
 The proof of Theorem~\ref{mainthm0} is complete.

\subsection*{Acknowledgments} The author thanks Yuto Nakajima and Mao Shinoda for fruitful discussions. This research was supported by the JSPS KAKENHI 25K21999 Grant-in-Aid for Challenging Research (Exploratory).


\begin{thebibliography}{99}

\bibitem{BM86} Anne Bertrand-Mathis,
{\it D\'eveloppement en base $\theta$; r\'epartition modulo un de la suite $(x\theta^n)_{n\geq0}$; langages cod\'es et $\theta$-shift}, Bull. Soc. Math. France
{\bf 114} (1986), 271--323.
\bibitem{Bow71} Rufus Bowen, {\it Entropy for group endomorphisms and homogeneous spaces}, Trans. Amer. Math. Soc. {\bf 153} (1971), 401--414.

\bibitem{Buz97} J\'er\^ome Buzzi, {\it Specification on the interval},
Trans. Amer. Math. Soc. {\bf 349} (1997), 2737–2754.

 \bibitem{CLR21} Leonard Carapezza, Marco L\'opez, and Donald Robertson, {\it Unique equilibrium
states for some intermediate beta transformations}, Stoch. Dyn. {\bf 21} (2021), no. 6,
25pp. Id/No 2150035.

\bibitem{DGS} Manfred Denker,  Christian Grillenberger, Karl Sigmund, {\it Ergodic theory on compact spaces}, Lecture Notes in Mathematics, {\bf 527} Springer-Verlag, Berlin-New York, 1976.



\bibitem{Fal} Kenneth Falconer, {\it Fractal geometry. Mathematical foundations and applications}, Second edition. John Wiley \& Sons, Inc., Hoboken, NJ, 2003.

\bibitem{Hal47} Marshall Hall, On the sum and product of continued fractions, Ann. of Math. {\bf 48} (1947), 966--993.

\bibitem{Hof80}  Franz Hofbauer, {\it On intrinsic ergodicity of piecewise monotonic transformations with positive entropy}, Israel J. Math. {\bf 34} (1979), no. 3, (1980), 213--237.


\bibitem{HHY17} Hui Hu, Zhihui Li, and Yueli Yu,
{\it A note on $(-\beta)$-shifts with the specification property}, Publ. Math. Debr.
{\bf 91} (2017), 123--131. 

 \bibitem{HKY93} Brian R. Hunt, Ittai Kan, and James A. Yorke, 
{\it When Cantor sets intersect thickly},
Trans. Amer. Math. Soc. {\bf 339} (1993), 869--888. 

\bibitem{IS09} Shunji Ito and Taizo Sadahiro, 
{\it Beta-expansions with negative bases}, {\it Integers} {\bf 9} (2009), 239--259.

\bibitem{KH} Anatole Katok and Boris Hasselblatt, {\it Introduction to the modern theory of dynamical systems}, Encyclopedia of Mathematics and its applications, {\bf 54}
Cambridge University Press, Cambridge, 1995. 

\bibitem{KLO16} Dominik Kwietniak, Martha {\L}\c{a}cka, and Piotr Oprocha, {\it A panorama of specification-like properties and their consequences},
{\it Dynamics and numbers}, 155--186, {\it Contemp. Math.} {\bf 669} (2016),
Amer. Math. Soc., Providence, RI.

\bibitem{LS12}
Lingmin Liao and Wolfgang Steiner, {\it Dynamical properties of the negative beta transformation}, Ergodic Theory Dynam. Systems {\bf 32} (2012), 1673--1690.

\bibitem{Mor98}
Carlos Gustavo T. de A. Moreira,  Sums of regular Cantor sets, dynamics and applications to number theory, Period. Math. Hungar. {\bf 37} (1998), 55--63.



\bibitem{New70} Sheldon E. Newhouse, {\it Nondensity of Axiom A(a) on $S^2$}, 
1970 Global Analysis (Proc. Sympos. Pure Math., Vol. XIV, Berkeley, Calif., 1968) pp. 191–202 Amer. Math. Soc., Providence, R.I. 57.48.

\bibitem{New79} Sheldon E. Newhouse, {\it The abundance of wild hyperbolic sets and non-smooth stable sets
for diffeomorphisms}, Inst. Hautes \'Etudes Sci. Publ. Math. 
{\bf 50} (1979), 101--151.

\bibitem{OS24} Mai Oguchi and Mao Shinoda, {\it Hausdorff dimension of the parameters for $(\alpha,\beta)$-shifts with the specification property},  Dyn. Syst. {\bf 39} (2024), 848--855.

\bibitem{PalTak93} Jacob Palis and Floris Takens, {\it Hyperbolicity and sensitive chaotic dynamics at
homoclinic bifurcations}.
Cambridge studies in advanced mathematics {\bf 35},
Cambridge University Press 1993.

\bibitem{Par60} William Parry, {\it On the $\beta$-expansions of real numbers}, Acta Math. Acad. Sci. Hungar. {\bf 11} (1960), 401--416.

\bibitem{Par64} William Parry, {\it Representations for real numbers}, Acta Math. Acad. Sci. Hungar. {\bf 15} (1964), 95--105.

\bibitem{R57}
Alfr\'ed R\'enyi, {\it Representations for real numbers and their ergodic properties}. Acta Math.
Acad. Sci. Hungar. {\bf 8} (1957), 477--493.

\bibitem{Sch97} J\"org Schmeling, {\it Symbolic dynamics for $\beta$-shifts
and self-normal numbers}. Ergodic Theory Dynam. Systems 
{\bf 17} (1997), 675--694.



\bibitem{Wil91} Robert F. Williams, {\it How big is the intersection of two thick Cantor sets?}, Continuum Theory and Dynamical Systems (Arcata, CA, 1989), pp. 163--175, Contemp. Math., 117 Amer. Math. Soc., Providence, RI, 1991.
\end{thebibliography}
\end{document}